\documentclass[11pt]{amsart}

\usepackage{amsthm}
\usepackage{latexsym}
\usepackage{amsmath}
\usepackage{latexsym}
\usepackage{amssymb,amscd}
\usepackage{amsfonts}
\usepackage{xspace}
\usepackage[frame,ps,matrix,arrow,curve,rotate]{xy}
\usepackage{amsfonts}
\usepackage{mathrsfs}

\newtheorem{theorem}{Theorem}[section]
\newtheorem{proposition}[theorem]{Proposition}

\newtheorem{question}[theorem]{Question}
\newtheorem{lemma}[theorem]{Lemma}
\newtheorem{corollary}[theorem]{Corollary}

\theoremstyle{remark}

\theoremstyle{definition}

\def\k{{\bf k}}
\def\P{{\bf P}}

\def\O{{\mathcal O}}

\def\Hilb{{\rm Hilb}}

\def\ker{\mathop{\rm ker}\nolimits}

\def\Hom{\mathop{\rm Hom}}

\def\image{\mathop{\rm image}}

\begin{document}

\title{Non-uniruledness results for spaces of rational curves 
in hypersurfaces}
\author{Roya Beheshti}
\date{}
\maketitle

\begin{abstract}
We prove that the sweeping components of the 
space of smooth rational curves in a smooth hypersurface of degree $d$ in $\P^{n}$ are not 
uniruled if $(n+1)/2 \leq d \leq n-3$. We also show that for any $e \geq 1$, 
the space of smooth rational curves of degree $e$ in a general hypersurface 
of degree $d$ in $\P^n$ is not uniruled roughly when $d \geq e \sqrt{n}$.
\end{abstract}

\section{Introduction}

Throughout this paper, we work over an algebraically closed field of 
characteristic zero $\k$. Let $X$ be a smooth hypersurface of degree $d$ in $\P^n$, and 
let $\Hilb_{et+1}(X)$ be the Hilbert scheme parametrizing  subschemes of 
$X$ with Hilbert polynomial $et+1$. Denote by $R_e(X)$ 
the closure of the open subscheme of $\Hilb_{et+1}(X)$ parametrizing smooth 
rational curves of degree $e$.  Following \cite{kodaira}, we call an irreducible component 
$R$ of 
$R_e(X)$ a {\it sweeping component} if the curves parametrized by 
its points sweep out $X$ or equivalently, 
if for a general point $[C]$ in $R$, $N_{C/X}$ is globally generated. 

In this paper, we consider the birational geometry of the sweeping components of 
$R_e(X)$, specifically, we are interested in the following question: for which values of $n, d$, and $e$, does $R_{e}(X)$ have non-uniruled sweeping components? 
A projective variety $Y$ of dimension $m$ is called uniruled if there is a variety $Z$ of dimension 
$m-1$ and a dominant rational map $Z \times \P^1 \dashrightarrow Y$. 
Our original motivation for this study comes from the question of whether or not general Fano hypersurfaces of low index are unirational.

We prove the following:

\begin{theorem}\label{main}
Let $X$ be any smooth hypersurface of degree $d$ in $\P^n$, $(n+1)/2 \leq 
d \leq n-3$. Then 
the sweeping components of $R_e(X)$ are all non-uniruled. 
\end{theorem}

Note that if $ d\leq n-1$, or if $d=n$ and $e\geq2$, $R_e(X)$ has at least one sweeping component. Also note that when $d \leq (n+4)/2$, $(d,n) \neq (3,3)$ and $X$ is general, $R_e(X)$ 
is irreducible (see \cite{hrs} and \cite{ij}), and it is conjectured that the same holds for general Fano hypersurface of dimension at least three \cite{ij}.

The proof of Theorem \ref{main} shows that when $d=n-2$ a sweeping component $R$ 
of $R_e(X)$ is non-uniruled if the normal bundle of a general 
curve parametrized by  $R$ 
is  balanced
(see Proposition  \ref{remark}). 
In the special case when $n=5$ and $d=3$, $R_e(X)$ is irreducible but it is not balanced. 
Modifying the 
proof of Theorem \ref{main}, we give a new proof of the following theorem:

\begin{theorem}[de~Jong-Starr \cite{ds}]\label{cubic}
If $X$ is a general cubic fourfold, then $R_e(X)$ is not 
uniruled when $e > 5$ is an odd integer, and the general fibers of 
the MRC fibration of any desingularization of $R_{e}(X)$ are at most 1-dimensional 
when $e > 4$ is an even integer.\\
\end{theorem}

The questions which remain are first, what happens when $d=n-1$, or $d=n$ and $e \geq 2$? 
Second, how small can $d$ be for $R_e(X)$ to be non-uniruled?
When $d=n$, the uniruledness of the sweeping subvarieties of $R_e(X)$ has been 
studied in \cite{roya-jason}. 
It is shown that if $d=n$ and $e \leq n$, a subvariety of $R_e(X)$ is 
non-uniruled if the curves parametrized by its points sweep out $X$ or a divisor in $X$. 

When  $d^2 \leq n$, $X$ is rationally simply connected (see \cite{s4} and 
\cite{ds1}),  and in particular $R_e(X)$ is uniruled. There are evidences which suggest that 
when $ d^2+d \geq 2n+2$,  $R_e(X)$ is non-uniruled for 
general $X$, but this is known to be true only for $e=1$. The following is one such evidence:

\begin{theorem}[J.~Starr \cite{kodaira}]\label{jason}
Let $X \subset \P^n$ be a general hypersurface of degree $d$, and let $\overline{\mathcal M}_{0,0}(X,e)$ be the
Kontsevich moduli space parametrizing rational curves of degree $e$ on $X$. 
If $d < min (n-6, \frac{n+1}{2})$ and $d^2+d \geq 2n+2$, then for every $e >0$ 
the canonical divisor of $\overline{\mathcal M}_{0,0}(X,e)$ is big.
\end{theorem}

One cannot directly  conclude from the theorem above that the coarse moduli space 
$\overline{M}_{0,0}(X,e)$ is also of general type when $d$ satisfies the 
inequalities given in the statement of the theorem.
In \cite{kodaira}, it is conjectured that when $d+e \leq n$ and $d$ is in the range given in Theorem \ref
{jason}, $\overline{M}_{0,0}(X,e)$ has at worst canonical singularities, and it is shown that assuming 
this conjecture, the above theorem implies $\overline{M}_{0,0}(X,e)$ and hence $R_e(X)$ 
are of general type when $d+e \leq n$.   

 In Section 4, we show: 

\begin{theorem}\label{lowdeg}
Let $X \subset \P^n$ ($n \geq 12$) be a general hypersurface of degree $d$, and let $m \geq 1$ be an integer. If a general smooth rational curve $C$ in $X$ is $m$-normal (i.e. 
the global sections of $\O_{\P^n}(m)$ maps surjectively to those of $\O_{\P^n}(m)|_C$),   
and if  
$$
d^2 + (2m+1)d \geq (m+1)(m+2)n+2,
$$ 
then $R_e(X)$ is not uniruled. 
\end{theorem}

Since 
every smooth rational curve of degree $e \geq 3$ in $\P^n$ is $(e-2)$-normal, we get:

\begin{corollary}\label{rem-lowdeg}
Let $X$ be a general hypersurface of degree $d$ in $\P^n$, $n \geq 12$. If $e\geq 3$ is an integer, and if     
$$d^2 + (2e-3)d \geq e(e-1)n+2,$$ 
then $R_e(X)$ is not uniruled. 
\end{corollary}

As it will be explained in the last section, we expect that better upper bounds exist on the regularity of general smooth rational curves contained in a general smooth hypersurface of degree $d$ in $\P^n$, so the bound in Corollary \ref{rem-lowdeg} could be possibly improved. 

\subsection*{Acknowledgments} 
The author would like to thank Izzet Coskun, Mohan Kumar, Mike Roth, and Jason Starr for many helpful conversations.

\section{A Consequence of Uniruledness}

In this section, we prove a proposition, analogous to the existence of free rational curves on 
non-singular uniruled varieties, for varieties whose spaces of smooth rational curves are 
uniruled. 

For a morphism $f: Y \to X$ between smooth varieties, by the {\it normal sheaf of $f$}  we 
will mean the cokernel of the induced map on the tangent bundles $T_Y \to f^*T_X$. 

If $Y$ is an irreducible projective variety, and if $\widetilde{Y}$ is a desingularization of $Y$, then the 
maximal rationally connected (MRC) fibration of $\widetilde{Y}$  is a smooth morphism 
$\pi: Y^0 \to Z$ from an open subset $Y^0 \subset \widetilde{Y}$ 
such that the fibers of $\pi$ are all rationally connected, and such that 
for a very general point $z \in Z$, any rational curve in $\widetilde{Y}$ 
intersecting $\pi^{-1}(z)$ is contained in 
$\pi^{-1}(z)$. The MRC fibration of any smooth variety exists and is unique up to birational 
equivalences \cite{kollar}. 

Let $Y$ be an irreducible projective variety, and assume the fiber of the MRC fibration of $\widetilde{Y}$ at a general point is 
$m$-dimensional. Then it follows from the definition that there is an irreducible component 
$Z$ of $\Hom(\P^1, Y)$ such that 
the map $\mu_1: Z \times \P^1 \to Y$ defined by $\mu_1([g], b) = g(b)$ is dominant  and the image of 
the map 
$\mu_2: Z \times \P^1 \times \P^1 \to Y \times Y
$ 
defined by 
$\mu_2([g],b_1,b_2) = 
(g(b_1),g(b_2))
$ 
has dimension $\geq \dim Y +m$.

\begin{proposition}\label{free}
Let $X \subset \P^n$ be a nonsingular projective variety. If an irreducible sweeping 
component $R$ of 
$R_e(X)$ is uniruled, then there 
exist a smooth rational surface $S$ with a dominant morphism $\pi: S \to \P^{1}$
and a generically finite 
morphism $f: S \to X$ with the following two properties:

\begin{enumerate}
\item[(i)] If $C$ is a general fiber of $\pi$, then $f|_C$ 
is a closed immersion onto a smooth curve parametrized 
by a general point of $R$. 
\item[(ii)] If $N_f$ denotes the normal sheaf of $f$, then $\pi_{*}N_f$ is 
globally generated. 
\end{enumerate}

\noindent Moreover, if the fiber of the MRC 
fibration of a desingularization of $R$ at a general point is 
at least $m$-dimensional, then there are such $S$ and $f$ with the additional property that 
$\pi_{*}N_f$ has an ample subsheaf of rank $ = m-1$. 
\end{proposition}

\begin{proof}
Let $U \subset R \times X$ be the universal family over 
$R$. Since $R$ is uniruled, there exist a quasi-projective 
variety $Z$ and a dominant morphism $\mu: Z \times \P^1 \to R$. 
Let $V \subset Z \times \P^1 \times X$ be the pullback of the 
universal family to $Z \times \P^{1}$, and denote by $q: V \to Z \times X$ and $p: V \to Z$ the projection maps.  

Consider a desingularization $g: \widetilde{V} \to V$, and let $\widetilde{q} = q \circ g$ and 
$\widetilde{p} = p \circ g$. Denote the fibers of $p$ and $\widetilde{p}$ over $z$ by $S$ and $\tilde{S}$ 
respectively. Let $f: S \to X$ be the restriction of $q$ to 
$S$, and let $\widetilde{f} = f \circ g: \widetilde{S} \to X$. 
Since $z$ is general, by generic smoothness, $\widetilde{S}$ is a smooth surface whose 
general fiber over $\P^1$ is a smooth connected rational curve. 
We claim that $\widetilde{S}$ and $\widetilde{f}$ satisfy the properties 
of the theorem. The first property is clearly satisfied.

To show the second property is satisfied, we consider the Kodaira-Spencer map associated to 
$\widetilde{V}$ at a general point $z \in Z$. Denote by 
$N_{\widetilde{q}}$ the normal sheaf of the map ${\widetilde{q}}$.
We get a sequence of maps 
$$
T_{Z,z} \to H^0(\widetilde{S}, \widetilde{p}^*T_Z|_{\tilde{S}}) 
\to H^0(\widetilde{S}, \widetilde{q}^*T_{X \times Z}|_{\widetilde{S}}) 
\to H^0(\widetilde{S}, N_{\widetilde{q}}|_{\widetilde{S}}).
$$
Let $b$ be a general point of $\P^1$. 
Composing the above map with the projection map $T_{Z \times \P^1, (z,b)} \to 
T_{Z,z}$, we get a map 
$
T_{Z\times \P^1, (z,b)} \to H^{0}(\widetilde{S}, N_{\widetilde{q}}|_{\widetilde{S}}).
$
Note that if $N_{\widetilde{f}}$ denotes the normal sheaf of $\widetilde{f}$, then 
$N_{\widetilde{q}}|_{\widetilde{S}}$ is naturally isomorphic to $N_{\widetilde{f}}$. 
Also, if $C$ is the fiber of $\pi: \widetilde{S} \to \P^1$ over $b$, then we have a short exact sequence 
$$
0 \to N_{C/\widetilde{S}}  \to N_{\widetilde{f}(C)/X} \to N_{\widetilde{f}}|_C \to 0.
$$
So we get a commutative diagram
$$\xymatrix{
T_{Z \times \P^{1}, (z, b)} \ar[d]^{d\mu_{(z,b)}} \ar[r] & T_{Z,z} \ar[r] & H^{0}(\widetilde{S}, 
N_{\widetilde{f}}) \ar[d] \\
T_{R, [\widetilde{f}(C)]}= H^0(\widetilde{f}(C), N_{\widetilde{f}(C)/X}) \ar[rr] &&H^0(C, N_{\widetilde{f}}|_C)
}
$$ 
Since $\mu$ is dominant, and since $R$ is sweeping and therefore generically smooth, 
$d\mu_{(z,b)}$ is surjective. Since the bottom row is also surjective, 
the map $H^{0}(\widetilde{S}, N_{\widetilde{f}}) 
\to H^{0}(C, N_{\widetilde{f}}|_{C})$ is surjective as well. Thus 
$\widetilde{\pi}_{*} N_{\widetilde{f}}$ is globally generated.

Suppose now that $R$ is uniruled and that the general fibers of the MRC fibration of $R$ 
are at least $m$-dimensional. Let $\dim R = r$. Then there exists a morphism 
$\mu_{1}: Z \times \P^{1} \to R$ such that the 
image of
$$
\mu_{2} : Z \times \P^{1} \times \P^{1} \to R \times R
$$
$$ 
\;\; \mu_{2} (z, b_{1}, b_{2}) = (\mu_{1}(z, b_{1}), \mu_{1}(z, b_{2}))
$$
has dimension $ \geq r + m$.
If $\widetilde{S}$ and $\widetilde{f}$ are as before,  and if 
$C_1$ and $C_2$ denote the fibers of $\pi$ over general points $b_1$ and $b_2$ of $\P^1$, then the image of the map 
$$d\mu_2 : T_{Z \times \P^{1} \times \P^{1}, (z, b_{1}, b_{2})} 
\to T_{R \times R, ([\widetilde{f}(C_1)], [\widetilde{f}(C_2)])} = H^0(C_1, 
N_{\widetilde{f}(C_1)/X}) \oplus 
H^0(C_2, N_{\widetilde{f}(C_2)/X}) $$ 
is at least ($r 
+ m$)-dimensional. The desired result now follows from the following commutative diagram 
$$
\xymatrix{ 
T_{Z \times \P^1 \times \P^1, (z, b_1, b_2)} \ar[d]^{(d\mu_2)_{(z,b_1,b_2)}} 
\ar[r] & T_{Z, z} \ar[r] & H^0(\widetilde{S}, N_{\widetilde{f}}) \ar[d] \\
 T_{R \times R, ([\widetilde{f}(C_1)], [\widetilde{f}(C_2)])}
\ar[rr] && H^0(C_1, N_{\widetilde{f}}|_{C_1}) 
\oplus H^0(C_2, N_{\widetilde{f}}|_{C_2}),
}
$$
and the observation that the kernel of the bottom row is 2-dimensional. 
\end{proof}

The above proposition will be enough for the proof of Theorem \ref{main},  but to prove Theorem \ref{cubic} in the even case, we will need a slightly stronger variant. Let $f: Y \to X$ be a morphism between smooth varieties, and let $N_f$ be the normal sheaf of $f$ 
$$
0 \to T_Y \to f^*T_X \to N_f \to 0.
$$
Suppose there is a dominant map $\pi: Y \to \P^1$, and let $M$ be the image of the map 
induced by $\pi$ on the tangent bundles $T_Y \to \pi^*T_{\P^1}$. 
Consider the push-out of the above sequence by the map $T_Y \to M$

 $$
\xymatrix{
0 \ar[r] &T_Y \ar[r] \ar[d] & f^*T_X \ar[r] \ar[d] & N_f \ar[r] \ar[d]^{=} & 0\\
0 \ar[r] & M \ar[r] \ar[d] & N_{f,\pi} \ar[r] \ar[d] & N_f \ar[r] & 0\\
& 0 & 0
 }
$$
The sheaf $N_{f,\pi}$ in the above diagram will be referred to as the {\it normal sheaf of  $f$ relative 
to $\pi$}. 

Property (ii) of Proposition \ref{free} says that $H^0(S, N_f) \to H^0(C, N_f|_C)$ is surjective. 
An argument parallel to the proof of Proposition \ref{free} shows the following: 

\begin{proposition}\label{stronger}
Let $X$ be as in Proposition \ref{free}. Then property (ii) can be strengthen as follows:

\begin{enumerate}
\item[(ii')] If $N_f$ denotes the normal sheaf of $f$, and if $N_{f,\pi}$ denotes the normal sheaf of $f$ 
relative to $\pi$, then the composition of the maps 
$$
H^0(S, N_{f,\pi}) \to H^0(C, N_{f,\pi}|_C) \to H^0(C, N_f|_C) 
$$ 
is surjective for a general fiber $C$ of $\pi$. 
\end{enumerate}

\noindent Moreover, if the general fibers of the MRC fibration of a desingularization of $R$ 
are at least $m$-dimensional, then there are $S$ and $f$ with properties (i) and (ii') such that the image of the map 
$$
H^0(S, N_{f,\pi}\otimes I_C) \to H^0(C, (N_f\otimes I_C)|_C)
$$
is at least ($m-1$)-dimensional.
\end{proposition}

\section{Proof of Theorem \ref{main}}
Throughout this section, $X$ will be a smooth hypersurface of degree $d \leq n-3$ in $\P^n$. 
Assume that a sweeping component $R$ of 
$R_e(X)$ is uniruled. Then we apply Proposition \ref{free} to show that $d <(n+1)/2$.  

\begin{proposition}\label{mainprop}
Let $S$ and $f$ be as in 
Proposition \ref{free}. If $C$ is a general fiber of 
$\pi: S \to \P^1$ and $I_C$ is the ideal sheaf of $C$ in $S$,  then the restriction map 
$$ H^{0}(S, f^{*}\O_{X}(2d-n-1) \otimes I_C^\vee) 
\to H^{1}(C, f^{*}\O_{X}(2d-n-1) \otimes I_C^\vee|_{C})$$
 is zero. 
 \end{proposition}
    
If the above map 
is zero, then 
$
H^0(S, f^*\O_X(2d-n-1))  = H^0(S, f^*\O_X(2d-n-1) \otimes I_C^\vee).
$
Thus, 
\begin{equation*}
\begin{split}
H^0(\P^1, \pi_*f^*\O_X(2d-n-1)) & = H^0(\P^1, \pi_*(f^*\O_X(2d-n-1) \otimes I_C^\vee))\\
& = H^0(\P^1, (\pi_*f^*\O_X(2d-n-1)) \otimes \O_{\P^1}(1)),
\end{split}
\end{equation*}
which is only possible if $H^0(\P^1, \pi_*f^*\O_X(2d-n-1))=0$. So 
$H^0(S, f^*\O_X(2d-n-1))=0$ and $d < (n+1)/2$.

\begin{proof}[Proof of Proposition \ref{mainprop}]

Let $\omega_S$ be the canonical sheaf of $S$. 
By Serre duality, it suffices to show that if $S$ and $f$ satisfy the properties of Proposition 
\ref{free}, then the restriction map 
$$
H^1(S, f^*\O_X(n+1-2d) \otimes \omega_S) \to H^1(C, f^*\O_X(n+1-2d) \otimes \omega_S|_C)
$$
is surjective. Let $N$ be the normal sheaf of the map $f: S \to X$, and let 
$N^\prime$ be the normal sheaf 
of the map $S \to \P^n$. 
Since the normal bundle of $X$ in $\P^{n}$ is 
isomorphic to  
$\O_{X}(d)$, we get a short exact sequence 
\begin{equation}\label{eq1}
0 \to N \to N'  \to f^*\O_X(d) \to 0.
\end{equation}
Taking the $(n-3)$-rd exterior power of this sequence, we get the following short exact sequence
$$ 
0 \to \bigwedge^{n-3} N \otimes f^*\O_X(-d) \to \bigwedge^{n-3} N' \otimes f^*\O_X(-d) 
\to \bigwedge^{n-4} N \to 0.
$$

\medskip

For an exact sequence of sheaves of $\O_S$-modules  
$ 0 \to E \to
F \to M \to 0
$ with $E$ and $F$ locally free  of ranks $e$ and $f$, there is a natural map of sheaves
 $$ 
 \bigwedge^{f-e-1} M \otimes \bigwedge^e E \otimes (\bigwedge^f F)^\vee \to M^{\vee}
 $$
 which is  defined locally at a point $s \in S$ as follows: assume 
		$\gamma_1, \dots, \gamma_{f-e-1} \in M_{s}$,  
		$\alpha_{1}, \dots, \alpha_{e} \in E_{s}$, and $\phi: 
		\bigwedge^{f}F_{s} \to O_{S,s}$; then for $\gamma \in M_s$, we set 
		$ \gamma_{f-e} = \gamma$, and we define the map to be 
		$$ \gamma  
		\mapsto \phi(\tilde{\gamma}_{1} \wedge \tilde{\gamma}_{2} 
		\wedge \dots \wedge 
		\tilde{\gamma}_{f-e} \wedge \alpha_{1} 
		\wedge \dots \wedge \alpha_{e})$$
		where $\tilde{\gamma}_{i}$ is any lifting of 
		$\gamma_{i}$ in $F_s$.
		Clearly, this map does not depend on the choice of the liftings, and thus 
		it is defined globally. 
So from the short exact sequence $0 \to T_S \to f^*T_X \to N \to 0$, we get a map 
$$  
\bigwedge^{n-4}N \to N^\vee \otimes f^*\O_X(n+1-d) \otimes \omega_S,
$$
and from the short exact sequence $0 \to T_S \to f^*T_{\P^n} \to N' \to 0$, we get 
a map 
$$ 
\bigwedge^{n-3} N' \otimes f^*O_X(-d) \to (N')^\vee \otimes 
    f^*O_X(n+1-d) \otimes \omega_S.
$$

With the choices of the maps we have made, 
the following diagram, whose bottom row is obtained from dualizing sequence (\ref{eq1}) and tensoring with $f^*\O_X(n+1-2d) \otimes \omega_S$, is commutative with exact rows

{\small{
$$
\xymatrix @ C=0.8pc{
0 \ar[r] & \bigwedge^{n-3}N \otimes f^*\O_X(-d) \ar[r] \ar[d] & \bigwedge^{n-3} N' \otimes 
f^*O_X(-d) \ar[r] \ar[d] & \bigwedge^{n-4}N  \ar[d] \ar[r] & 0 \\
0 \ar[r] & f^*\O_X(n+1-2d) \otimes \omega_S \ar[r] & (N')^\vee \otimes f^*\O_X(n+1-d) 
\otimes \omega_S \ar[r] & N^\vee \otimes f^*\O_X(n+1-d) \otimes \omega_S  \ar[r] 
& 0 
}
$$
}}

\bigskip 

Since the cokernel of the first vertical map restricted to $C$ is a torsion sheaf, to show the assertion, 
it will suffice to show that the map 
$$
H^1(S, \bigwedge^{n-3}N \otimes f^*\O_X(-d)) \to 
H^1(C, \bigwedge^{n-3}N \otimes f^*\O_X(-d)|_C)
$$ 
is surjective. Applying the long exact sequence of 
cohomology to the top sequence, the surjectivity assertion follows if we show that 

\begin{enumerate}
\item[(1)]$H^0(S, \bigwedge^{n-4}N) \to H^0(C, \bigwedge^{n-4}N|_C)$ is surjective.
\item[(2)] $H^1(C, \bigwedge^{n-3}N' \otimes f^*\O_X(-d)|_C)=0$.
\end{enumerate}

\noindent Since $R$ is a sweeping component, $N_{f(C)/X}$ is globally 
generated, so $f^*T_X|_C$ and $N|_C$ are globally generated as well. Thus we can conclude (1) from the 
assumption that $\pi_*N$ is globally generated and hence $H^0(S, N) \to H^0(C, N|_C)$ is surjective. 

To show (2), note that there is a surjective map 
$f^*\O_{\P^n}(1)^{\oplus n+1} \to f^*T_{\P^n}$, so we get  a surjective map 
$
f^*\O_{\P^n}(1)^{\oplus n+1} \to N'.
$
Taking the 
$(n-3)$-rd exterior power, and then tensoring with $f^{*}\O_X(-d)$, 
we get a surjective map 
$$
 f^{*}\O_{\P^n}(n-3-d)^{\oplus {n+1 \choose{n-3}}} \to \bigwedge^{n-3}N' \otimes 
f^{*}\O_{X}(-d).
$$ 
Restricting to $C$, 
since $ n-3-d \geq 0$, we have $H^1(C, \bigwedge^{n-3}N' \otimes f^*\O_X(-d)|_C) = 0$.
\end{proof}

We conclude this section with a result on the case $d=n-2$. Let $C$ be a smooth rational 
curve of degree $e$ in $\P^n$ whose normal bundle 
$N_{C/{\P^n}}$ is globally generated. If we write 
$$
N_{C/{\P^n}} = \O_C(a_1) \oplus \dots \oplus \O_C(a_{n-1}), 
$$ 
then $\sum_{1 \leq i \leq n-1} a_i = e(n+1)-2$. 
We say $N_{C/\P^n}$ is {\em balanced} if $a_i + a_j < 3e$ for every $i \neq j$, and we call 
an irreducible component $R$ of $R_e(X)$ a {\em balanced component} if $N_{C/\P^n}$ is balanced for a 
general curve 
$C$ parametrized by $R$. 

If $C$ is balanced and $d=n-2$, then $H^1(C, \bigwedge^{n-3} N_{C/\P^n} \otimes \O_{\P^n}(-d)|_C) =0$. Thus 
if $N'$ is as in the proof of Theorem \ref{main}, we have  
$$
H^1(C, \bigwedge^{n-3} N' \otimes f^*\O_X(-d) |_C)=0.
$$ 
The proof of Theorem \ref{main} then shows that 

\begin{proposition}\label{remark}
Suppose that $X$ is a smooth hypersurface of degree $d=n-2$ in $\P^n$. 
Let $R$ be a sweeping  component of $R_e(X)$ and $C$ 
a curve parametrized by a general point of $R$.  If $N_{C/{\P^n}}$ is balanced, then $R$ is not uniruled. 
\end{proposition}
 
It might be true that a general hypersurface of degree $n-2$ in 
 $\P^n$ has a balanced component when $n \geq 6$ and $e \geq 1$. 
 If $n=5$ and $d=3$, then the normal bundle of the curves parametrized by general points of 
 $R_e(X)$ are not balanced when $e \geq 6$ is an even integer (see Proposition \ref{normal}). 

\section{Cubic Fourfolds}
In this section, we prove Theorem \ref{cubic}, but before giving the proof, we 
first briefly explain the idea  of the proof given in \cite{ds}.

Let $X \subset \P^5$ be a general hypersurface of degree 3.
Then by [1, Proposition 2.4], $R_e(X)$ is irreducible, and if we denote by $\overline{\mathcal M}_{0,0}(X,e)$ the Kontsevich moduli space of stable maps of 
degree $e$ from curves of genus zero to $X$ and by $\overline{M}_{0,0}(X,e)$ the corresponding 
coarse moduli space, then $\overline{M}_{0,0}(X,e)$ is birational to $R_e(X)$. Let $\widetilde{M}$ 
a desingularization of $\overline{M}_{0,0}(X,e)$. 
\begin{theorem}[\cite{ds}, Theorem 1.2]\label{ds}
Let $X \subset \P^5$ be a general cubic hypersurface. 
There is a canonical section $\omega_e \in H^0(\widetilde{M}, 
\Omega^2_{\widetilde{M}})$ with 
the following property:

\noindent  (a) If $e$ is odd, $e \geq 5$, and if $p$ is a general point of 
$\widetilde{M}$, then $\omega_e$ induces a non-degenerate pairing on $T_{\widetilde{M},p}$. 

\noindent (b) If $e$ is even, $e \geq 6$,  and if $p$ is a general point of 
$\widetilde{M}$, then the linear map $T_{\widetilde{M}, p} 
\to T_{\widetilde{M},p}^\vee$ 
induced by $\omega_e$ has a 1-dimensional kernel.
\end{theorem}

If $Y$ is a non-singular projective variety, and if $Y$ has a non-zero 2-form $\omega$ such that for a general point $p \in Y$, the kernel of the map 
$$ T_{Y,p} \to T_{Y,p}^\vee
$$
induced by the restriction of $\omega$ to $p$ has dimension at most $m$, then 
the fiber of the MRC fibration of $Y$ at a general point 
is at most $m$-dimensional \cite[Lemma 1.4]{ds}. So 
the above theorem implies Theorem \ref{cubic}. 

The proof of Theorem \ref{ds} is based on a general construction of differential forms on any  desingularization of $\overline{M}_{0,0}(X,e)$. Pulling back forms to the universal 
curve over $\overline{\mathcal M}_{0,0}(X,e)$, and then integrating along the fibers, one gets 
a linear  map 
$$
H^{i+1}(X, \Omega_X^{j+1}) \to H^{i}(\overline{\mathcal M}_{0,0}(X,e), \Omega^{j}_{\overline{\mathcal M}_{0,0}(X,e)}).
$$
When $i=0$, this map gives $j$-forms on the moduli stack. Invoking the existence of a 
trace map then gives $j$-forms on any desingularization of the coarse moduli 
space (\cite{ds}, Proposition 3.6).  For a cubic 
threefold, $H^1(X, \Omega_X^3)$ is 1-dimensional, so from the above construction, 
one gets a natural 2-form $\omega_e$ on 
$\widetilde{M}$. 

The next step is to compute the dimension of the kernel of  $\omega_e$ restricted to a general point $p \in \widetilde{M}$. Note that if $C \subset X$ is the rational curve 
parametrized by $p$, then $T_{\widetilde{M},p} = H^0(C, N_{C/X})$, so $\omega_{e,p}$ 
gives a map $\delta: \bigwedge^2H^0(C, N_{C/X}) \to \k$. 

To prove Theorem \ref{ds}, for a general smooth rational curve $C$ of degree $e$ on $X$, $N_{C/P^n}$ 
and $N_{C/X}$ are computed, and then, the corresponding pairing is described. 

The proof of Theorem \ref{cubic} given in this section has some similarities with the proof outlined above, 
but our method is local, and that enables 
us to avoid the technicalities and most of the computations involved in the proof presented in 
\cite{ds}. For our  purpose, it is enough to compute $N_{C/{\P^n}}$ for a general rational curve 
$C$ of degree $e$ on $X$. 

\begin{proposition}[\cite{ds}, 7.1] \label{normal}
Let $X$ be a general cubic fourfold, and let $C$ be a general smooth 
rational curve of degree $e\geq 5$ on $X$. Then 
$$N_{C/\P^5}=
\begin{cases}
\O_C(\frac{3e-1}{2})^{\oplus 4} \;\;\;\;\;\;\;\;\;\;\;\;\;\;\;\;\;\;\;\;\;\;\;\;\; \text{ if  }  e \text{ is odd},\\
\O_C(\frac{3e}{2})^{\oplus 2} \oplus \O_C(\frac{3e-1}{2})^{\oplus 2} \;\;\;\;\;  \text{ if } e  \text{ is even.}
\end{cases}
$$
\end{proposition}

\bigskip

\begin{proof}[Proof of Theorem \ref{cubic}] 
When $e \geq 5$ is odd, $R_e(X)$ is balanced by Proposition \ref{normal}, so the assertion 
follows from Proposition \ref{remark}.
 
Let now $e \geq 6$ be an even integer, and assume on the contrary that general fibers of the 
MRC fibration of $R_e(X)$ are at least 2-dimensional. Let $S$ and $f$ be as in Proposition 
\ref{stronger}, and let $C$ be a general fiber of $\pi$. Set $N=N_f$ and $Q = N_{f,\pi}$. Then the following properties are satisfied.

\begin{enumerate}
\item[(i)] The composition of the maps 
$$
H^0(S, Q) \to H^0(S,Q|_C) \to H^0(C, N|_C)
$$ 
is surjective.
\item[(ii)] The composition of the maps 
$$
H^0(S, Q \otimes I_{C}) \to 
H^0(C, Q \otimes I_C|_C) \to H^0(C, N \otimes I_{C}|_C)
$$ 
is non-zero.
\end{enumerate}

\noindent We show these lead to a contradiction. 

Let $Q'$ be the normal sheaf of the map $S \to \P^5$ relative to $\pi$. 
We have $Q|_C = N_{C/X}$ and $Q'|_C = N_{C/\P^5}$. 
Since $N_{X/\P^5}=\O_X(3)$, there is a short exact sequence 

\begin{equation}\label{seqseq}
0 \to Q \to Q^\prime \to f^*\O_X(3) \to 0.
\end{equation}
Taking exterior powers, we obtain the following short exact sequence 
\begin{equation}
 0 \to \bigwedge^2Q \otimes f^*\O_X(-3) \to \bigwedge^2 Q^\prime \otimes f^*\O_X(-3) 
\to Q \to 0.
\end{equation}
Since this sequence splits locally, its restriction to $C$ is also a short exact sequence 
\begin{equation}\label{seq3}
 0 \to \bigwedge^2Q \otimes f^*\O_X(-3)|_C \to \bigwedge^2 Q^\prime \otimes f^*\O_X(-3)|_C 
\to Q|_C \to 0.
\end{equation}
Let $V$ be the image of the restriction map 
$$
\alpha: H^1(S, \bigwedge^2Q \otimes f^*\O_X(-3)) \longrightarrow H^1(C, \bigwedge^2Q \otimes 
f^*\O_X(-3)|_C).
$$
We will show that our assumptions imply that $V$ is of codimension at least 2 and that 
the image of the 
boundary map $H^0(C, Q|_C) \to H^1(C, \bigwedge^2Q \otimes f^*\O_X(-3)|_C)$ is a subset of $V$.  
This is not possible since by Proposition \ref{normal},
\begin{equation*}
\begin{split}
H^1(C, \bigwedge^2 Q^\prime \otimes f^*\O_X(-3)|_C) &= H^1(C, \bigwedge^2 N_{C/\P^5} \otimes 
f^*\O_X(-3)|_C) \\ &= H^1(C, \O_C(-2)\oplus \O_C(-1)^{\oplus 2} 
\oplus \O_C)\\
&= {\bf k} .
\end{split}
\end{equation*}

\begin{lemma}\label{trivial}
The kernel of the map $f^*T_X \to Q$ is a line bundle which contains 
$\bigwedge^2 T_S \otimes \pi^*\Omega_{\P^1}$ as a subsheaf. 
\end{lemma}

\begin{proof}
The kernel of $f^*T_X \to Q$ is equal to the kernel of the map induced by $\pi$ on the tangent bundles 
$T_S \to \pi^*T_{\P^1}$ which we denote by $F$ 
$$ 0 \to F \to T_S \to \pi^*T_{\P^1}.$$
Since $F$ is reflexive, it is locally free on $S$, and it is clearly of rank 1. Also the composition of the maps 
$$ \bigwedge^2T_S \otimes \pi^*\Omega_{\P1} \to \bigwedge^2T_S \otimes \Omega_S = T_S \to \pi^* 
T_{\P^1}$$
is the zero-map. 
So $\bigwedge^2 T_S \otimes \pi^*\Omega_{\P^1}$ is a subsheaf of $F$. 
\end{proof}

\bigskip

Given a section $r \in H^0(C, Q \times I_C|_C) $, we can define a map 
$$
\beta_r: H^1(C,\bigwedge^2Q \otimes f^*\O_X(-3)|_C)  \longrightarrow H^1(C, \omega_S|_C)= \k
$$
as follows. The lemma above implies there is  a generically injective map 
$$
\Psi: \bigwedge^3Q \otimes f^*\O_X(-3) \otimes I_C \to \omega_S \otimes \pi^*T_{\P^1} \otimes I_C.
$$
Restricting to $C$, we get a map 
$$
\Psi|_C : (\bigwedge^3Q \otimes f^*\O_X(-3) \otimes I_C)|_C \to 
\omega_S|_C.
$$
Also, $r$ gives a map
$$
\Phi_r: \bigwedge^2Q \otimes f^*\O_X(-3)|_C \stackrel{\wedge r}{\longrightarrow} \bigwedge^3Q \otimes f^*\O_X(-3) \otimes I_C|_C,
$$
and we define $\beta_r$ to be the map induced by the composition $\Psi|_C\circ \Phi_r$. Note 
that $\beta_r$ is non-zero if $r \neq 0$.

\begin{lemma}\label{dif}
For  sections $r, r^\prime \in H^0(C, Q\otimes I_C|_C)$, $\ker(\beta_r) = \ker(\beta_{r^\prime})$ 
if and only if $r$ and $r^\prime$ are scalar multiples of each other.
\end{lemma}

\begin{proof}
By Serre duality,  
it is enough to show that the images of the maps 
$$
\xymatrix{
H^0(C, I_C^\vee|_C) = H^0(C, \omega_S^\vee|_C \otimes \Omega_C)  \ar@<0.7ex>[r]^-{\beta_r^\vee} 
\ar@<-0.7ex>[r]_-{\beta_{r'}^\vee} & 
 H^0(C, (\bigwedge^2Q^\vee \otimes f^*\O_X(3))|_C \otimes 
\Omega_C)
}
$$
are the same if and only if $r$ and $r'$ are scalar multiples of each other. 
Since $Q|_C = N_{C/X}$, we have $\bigwedge^3 Q|_C = \bigwedge^3 N_{C/X} = f^*\O_X(3) \otimes \Omega_C$, so 
$$
(\bigwedge^2Q^\vee \otimes f^*\O_X(3))|_C \otimes \Omega_C 
= Q|_C,
$$ 
and the map
$$
\beta_r^\vee: H^0(C, I_C^\vee|_C) \to H^0(C, Q|_C)
$$ 
is simply given by $r$. Similarly, $\beta_{r'}^\vee$ is given by $r'$. The lemma follows.
\end{proof}

Recall that by definition, we have a short exact sequence 
$$ 
0 \to \pi^*T_{\P^1}|_C  \to Q|_C \to N|_C \to 0,
$$
and $ \pi^*T_{\P^1}|_C = I_C^{-1}|_C$. If we tensor this sequence with $I_C|_C$, we get the following short exact sequence
$$
0 \to \O_C \to Q \otimes I_C|_C \to N \otimes I_C|_C \to 0.
$$
Let $i$ be a non-zero section in the image of $H^0(C, \O_C) \to H^0(C, Q \otimes I_C|_C)$. Then 
$i$ induces a map 
$$
\beta_i:  H^1(C,\bigwedge^2Q \otimes f^*\O_X(-3)|_C)  \longrightarrow
H^1(C, \omega_S|_C)= \k
$$ 
as above. Let $$\gamma: H^0(C, Q|_C) \to H^1(C, \bigwedge^2Q \otimes f^*\O_X(-3)|_C)$$ be the 
connecting map in sequence (\ref{seq3}). 

\begin{lemma}\label{lem-i} 
We have
\begin{enumerate}
\item[(a)] $V \subset \ker \beta_i$.
\item[(b)] $ \image (\gamma)\subset \ker \beta_i$.
\end{enumerate}
\end{lemma}

\begin{proof}

(a) From the short exact sequence $0 \to T_S 
\to f^*T_X \to N \to 0$, we get a map $$\bigwedge^2 N \otimes f^*\O_X(-3) \to \omega_S,
$$ 
and so, there is a commutative diagram
{\small
$$ 
\xymatrix{
H^1(S,\bigwedge^2Q \otimes f^*\O_X(-3)) \ar[r] \ar[d]^{\alpha} & 
H^1(S,\bigwedge^2N \otimes f^*\O_X(-3)) \ar[r] & H^1(S,\omega_S) =0 \ar[d]\\
H^1(C, \bigwedge^2Q \otimes f^*\O_X(-3)|_C) \ar[rr]^{\beta_i} & & H^1(C, \omega_S|_C).
}
$$

}

(b) Applying the long exact sequence of cohomology to the exact sequence 
$$
0 \to \bigwedge^2N \otimes f^*\O_X(3) \to \bigwedge^2 N' \otimes f^*\O_X(-3) \to N \to 0,
$$ 
we get a map 
$$  H^0(C, N|_C) \to H^1(C, \bigwedge^2N \otimes f^*\O_X(-3)|_C).$$
The map 
$\beta_i \circ \gamma$ factors through 
$$
H^0(C, Q|_C) \to H^0(C, N|_C) \to H^1(C, \bigwedge^2N \otimes f^*\O_X(-3)|_C) \to H^1(C, \omega_S|_C),
$$
and we have a commutative diagram 
$$ 
\xymatrix{
& H^0(S, N) \ar[r] \ar@{>>}[d]& 
H^1(S,\bigwedge^2N \otimes f^*\O_X(-3)) \ar[r] & H^1(S,\omega_S) =0 \ar[d]\\
H^0(C,Q|_C) \ar[r]  &H^0(C, N|_C) \ar[rr] & & H^1(C, \omega_S|_C).
}
$$
Thus we can conclude the assertion from the fact that the restriction map $H^0(S, N) \to H^0(C, N|_C)$ is surjective, and  so the image of the composition of the above maps 
is contained in the image of the restriction map 
$H^1(S, \omega_S ) \to H^1(C, \omega_S|_C)$ which is zero.
\end{proof}

In the following lemma we prove similar results for the sections of $Q \otimes I_C|_C$ which are 
restrictions of global sections of $Q \otimes I_C$.
 
\begin{lemma}\label{lem-r}
If $\tilde{r} \in H^0(S, Q\otimes I_C)$, and if $r = \tilde{r}|_C$, then we have 
\begin{itemize}
\item[(a)] $ V \subset \ker(\beta_r).$
\item[(b)] $ \image(\gamma)  \subset \ker (\beta_r).$
\end{itemize}
\end{lemma}

\begin{proof}
(a) Since $r$ is the restriction of $\tilde{r}$ to $C$, we get a commutative diagram
$$
\xymatrix{
H^1(S,\bigwedge^2Q \otimes f^*\O_X(-3)) \ar[r]^----{\wedge \tilde{r}} \ar[d]^{\alpha} & 
H^1(S, \bigwedge^3Q \otimes f^*\O_X(-3) \otimes I_C) \ar[r]^{\Psi} &
H^1(S, \omega_S \otimes \pi^*T_{\P^1} \otimes I_C)=0 \ar[d] \\
H^1(C,\bigwedge^2Q \otimes f^*\O_X(-3)|_C)  \ar[rr]^---{\beta_r} &&
H^1(C, \omega_S|_C).
}
$$ 

(b) Consider the short exact sequence 
$$ 
0 \to I_C^{-1}|_C \to Q|_C \to N|_C \to 0.
$$
Since by property (i), $H^0(S, Q) \to H^0(C, N|_C)$ is surjective, to prove the statement, it is enough to show 
that for any non-zero $u$ in the image of $H^0(C, I_C^{-1}|_C) \to H^0(C, Q|_C)$, we have $\gamma(u) 
\in \ker \beta_r$.

Consider the diagram
$$
\xymatrix
{
H^0(C, Q|_C) \ar[r]^--{\gamma} \ar@/_/[d]_{\wedge i}  \ar@/^/[d]^{\wedge r} & H^1(C, \bigwedge^2Q \otimes f^*\O_X(-3)|_C) \ar@/_/[d]_{\wedge i}  \ar@/^/[d]^{\wedge r} \ar@/_/[dr]_-------------------{\beta_i} \ar@/^/[dr]^{\beta_r} \\
H^0(C, \bigwedge^2 Q \otimes I_C|_C) \ar[r]_--{\lambda} & 
H^1(C, \bigwedge^3Q \otimes f^*\O_X(-3) \otimes I_C|_C) \ar[r]_--{\psi} & H^1(C, \omega_S|_C)
}
$$
where $\lambda$ is obtained from applying the long exact sequence of cohomology to 
the third wedge power of sequence (\ref{seqseq}), and $\psi $ is induced by the map 
$\Psi |_C$. Then we have 
\begin{equation*}
\begin{split}
 \beta_r \circ \gamma (u)  & = \psi \circ \lambda (u \wedge r) \\
 & = \psi \circ \lambda (-r \wedge i)  \;\; \text{ (up to a scalar factor)} \\
& = \beta_i \circ \gamma (-r)\\
&  = 0,
\end{split}
\end{equation*}
where the last equality comes from the fact that $\gamma(H^0(C, Q|_C)) \subset \ker \beta_i$ 
by part (b) of Lemma \ref{lem-i}. 
\end{proof}

Let now $\tilde{r}_0 \in H^0(S, Q \otimes I_C)$ be so that its image in $H^0(C, N\otimes 
I_C|_C)$ is non-zero. Such $\tilde{r}_0$ exists by property (ii). Then $r_0 := \tilde{r}_0|_C$ defines a map $\beta_{r_0}$. According to Lemma \ref{dif}, $\ker \beta_{r_0} \neq \ker \beta_i$, so 
the codimension of $\ker \beta_i \cap \ker \beta_{r_0}$ is at least 2. 
On the other hand, by the previous lemmas,   
$\image(\gamma) \subset \ker \beta_i \cap \ker \beta_{r_0} $. This is a contradiction since 
$\dim H^1(C, \bigwedge^2 Q^\prime \otimes f^*\O_X(-3)|_C)=0$.

\end{proof}

\bigskip

\section{Low Degree Hypersurfaces}

Let  $X \subset \P^n$ be  a 
general hypersurface of degree $d \leq n/2$. Then by the main theorem of \cite{hrs}, $R_e(X)$  is 
irreducible of dimension $e(n+1-d)+(n-4)$. If $d^2 \leq n$, then by \cite{ds1} and \cite{s4}, $X$ is rationally simply connected. This means that 
for every $e \geq 2$, the evaluation morphism 
$$ev: \overline{\mathcal M}_{0,2}(X,e)  \to X \times X$$
is surjective and a general fiber of $ev$ is irreducible and rationally connected. In particular, 
$R_e(X)$ is rationally connected for $e \geq 2$. If $e=1$, then $R_1(X)$ is the Fano variety 
of lines on $X$ which is  rationally connected if and only if ${{d+1}\choose{2}} \leq n$ 
\cite[V.4.7]{kol}. 

As it was mentioned in the introduction, we expect that when $d^2$ is large compared to $n$, $R_e(X)$ is not uniruled. Here we outline a possible approach based, in part,  on the results of 
the previous sections. 

If $R_{e}(X)$ is uniruled, then 
there are $S$ and $f$ with the two 
properties given in Proposition  \ref{free}. We can assume that the pair $(S,f)$ is minimal in the sense 
that a component of a fiber of $\pi$ which is contracted by 
$f$ cannot be blown down. Let $N$ be the normal sheaf of $f$, 
and let $C$ be a general fiber of $\pi$ with ideal sheaf $I_{C}$ in $S$.  

Denote by $H$ the pullback of a hyperplane in $\P^n$ to $S$, and denote by $K$ a 
canonical divisor on $S$. From the exact sequences 
$0 \to T_S \to f^*T_X \to N \to 0$ and $0 \to f^*T_X \to f^*T_{\P^n} \to f^*\O_{\P^n}(d) \to 0$ 
we get  
\begin{equation*}
\begin{split}
\chi(N \otimes I_{C})  & =  (n+1) \chi (f^*\O_{\P^n}(1) \otimes I_C) - \chi (f^*\O_{\P^n}(d) \otimes I_C) - 
\chi (I_C) - \chi (T_S \otimes I_C) \\
& = (n+1) (\frac{(H-C)\cdot (H-C-K)}{2} +1) - \frac{(dH-C)\cdot(dH-C-K)}{2} -1\\
& \phantom{a} - \frac{-C \cdot (-C-K)}{2}-1
-(2K^2-14)\\
& =  \frac{(n+1-d^{2})}{2} H^{2} - 
\frac{(n+1-d)}{2} H \cdot K  -2K^{2}   - (n+1-d) e + 14.
\end{split}
\end{equation*}

We claim that $2H + 2C + K$ is base-point free and hence 
has a non-negative 
self-intersection number. By the main theorem 
of \cite{reider}, if 
$2H + 2C +K$ is not base point free, then there exists an effective 
divisor $E$ such that either 
$$ 
(2H +2C) \cdot E =1, E^{2} = 0 \;\; \text{ or} \;\; (2H+2C) \cdot E = 0, E^{2} = -1.
$$
The first case is clearly not possible. In the 
second case, $H \cdot E = 0$, and $C \cdot E =0$. 
So $E$ is a component of one of the fibers of $\pi$ which 
is contracted by $f$ and which is a $(-1)$-curve. This contradicts 
the assumption that $(S,f)$ is minimal. Thus $(2H+2C+K)^{2} \geq 0$. 
Also, since $H^1(S, f^*\O_X(-1))=0$, 
$$H \cdot (H+K)  = 2 \chi (f^*\O_X(-1)) - 2 \geq -2,$$ 
so we can write 

\begin{eqnarray*}
\chi(N \otimes I_{C}) & = & \frac{2n+2-d^2-d}{2} H^{2} - (n-d-15)(e-1)-2\\
& & - 2(2H+2C+K)^{2} - \frac{n-d-15}{2}(H \cdot (H+K)+2)\\
& \leq & \frac{2n+2-d^2-d}{2} H^{2} - (n-d-15)(e-1)-2,
\end{eqnarray*}
and therefore $\chi(N \otimes I_{C})$ is negative when $ d^{2}+d 
\geq 2n+2$ and $n \geq 30$.  

The Leray spectral sequence gives a short exact sequence 
$$
0 \to H^{1}(\P^1, \pi_{*} (N \otimes I_C)) \to 
H^1(S, N \otimes I_{C}) \to  
H^{0}(\P^{1}, R^{1}\pi_{*} (N \otimes I_C)) 
\to 0,
$$
and by our assumption on $S$ and $f$, $H^{1}(\P^1, \pi_{*} (N \otimes I_C))=0$. 
If we could choose $S$ such that $H^{0}(\P^{1}, R^{1}\pi_{*} (N \otimes I_C)) =0$, 
then we could conclude that $\chi(N \otimes I_{C}) \geq 0$ and hence $R_e(X)$ could 
not be uniruled for $ d^{2}+d 
\geq 2n+2$ and $n \geq 30$.

We cannot show that for a general $X$, a minimal pair $(S,f)$ as in Proposition \ref{free} 
can be chosen so that  $H^{0}(\P^{1}, R^{1}\pi_{*} (N \otimes I_C)) =0$. 
However, we prove that when $X$ is general, and $(S,f)$ is minimal, 
for every $t \geq 1$, 
$$H^{0}(\P^{1}, R^{1}\pi_{*} (N \otimes I_C \otimes f^*\O_X(t))) =0.$$
We also show that if $t \geq 0$ and $f(C)$ is $t$-normal,  
$$H^{1}(\P^1, \pi_{*} (N \otimes I_C \otimes f^*\O_X(t)) )=0.$$
These imply that $\chi (N \otimes I_C \otimes f^*\O_X(t))$ is non-negative when $X$ is general 
and $f(C)$ is $t$-normal.  To finish the proof of Theorem \ref{lowdeg}, 
we compute $\chi (N \otimes I_C \otimes f^*\O_X(t))$ directly and show that 
it is negative when the inequality in the statement of the theorem holds. 

\begin{proof}[Proof of Theorem \ref{lowdeg}]
Let $X$ be a general hypersurface of degree $d$ in $\P^n$. 
If $R_e(X)$ is uniruled, then there are $S$ and $f$ as in 
Proposition \ref{free}. Assume the pair $(S,f)$ is minimal. 
Let $N$ be the normal sheaf of $f$, and let $C$ be a general fiber of $\pi$.   
Then $H^0(S, N) \to H^0(C, N|_C)$ is surjective. The restriction 
map $H^0(S, f^*\O_X(m)) \to H^0(C, f^*\O_X(m)|_C)$ is also surjective 
since $f(C)$ is $m$-normal, so the restriction map $H^0(S, N \otimes f^*\O_X(m)) 
\to H^0(C, N \otimes f^*\O_X(m)|_C)$ is surjective as well. Therefore, 
$$H^1(\P^1, \pi_*(N \otimes f^*\O_X(m) \otimes I_C)) = 0.$$

Now let $C$ be an arbitrary fiber of $\pi$, and let $C^0$ be an irreducible component of $C$. Then by 
Proposition \ref{global}, $f^*(T_X(t))|_{C^0}$ 
is globally generated for every $t \geq 1$, and hence $N \otimes f^*\O_X(t)|_{C^0}$ is globally generated too. So Proposition \ref{h1} shows that for every $t \geq 1$ 
$$H^0(\P^1, R^1\pi_*(N \otimes f^*\O_X(t)\otimes I_C))=0.$$ 
By the Leray spectral sequence, 

\begin{equation*}
\begin{split}
H^1(S, N \otimes f^*\O_X(m)\otimes I_C) &= H^1(\P^1, \pi_*(N \otimes f^*\O_X(m) \otimes I_C)) \\
& \phantom{a} \oplus H^0(\P^1, R^1\pi_*(N \otimes f^*\O_X(m) \otimes I_C)) \\
& =0,
\end{split}
\end{equation*}

\noindent and therefore, 
$$
\chi(N \otimes f^*\O_X(m) \otimes I_C) \geq 0.
$$

We next 
compute $\chi(N \otimes f^*\O_X(m) \otimes I_C)$.
For an integer $t \geq 0$, set 
$$a_t = \chi (N \otimes I_C \otimes f^*\O_X(t)).
$$ 
We have 
\begin{equation*}
\begin{split}
a_t & = \chi (N \otimes I_C) + \frac{2t(n+1-d) +t^2(n-3)}{2}H^2 - \frac{t(n-5)}{2}H \cdot K -t(n-3)e.
\end{split}
\end{equation*}
So 
$$ a_t = \frac{b_t}{2} \,H^2 + \frac{c_t}{2} \, H \cdot K -2K^2 + d_t,$$
where $$b_t = (n+1-d^2) +2t(n+1-d)+t^2(n-3),$$ $$c_t= -(n+1-d)-t(n-5),$$ 
and 
$$d_t = -t(n-3)e-(n+1-d)e+14.$$
A computation similar to the computation in the beginning of this section shows that 

\begin{equation*}
\begin{split}
a_t & = \frac{b_t-c_t}{2}H^2 -2(2H+2C+K)^2+ \frac{c_t+16}{2} (H \cdot (H+K)+2) + (d_t-c_t -32 +16e)\\
& \leq \frac{b_t-c_t}{2}H^2 + (d_t-c_t-32+16e).
\end{split}
\end{equation*}

\noindent Since 
$$d_t-c_t -32 + 16e = -(e-1)(n-15-d+t(n-3)) -2t -2,$$ 
and since 
$n-15-d+t(n-3) \geq 2n-d-18 \geq 0$ for $t \geq 1$ and $n \geq 12$, we get 

$$ a_t < \frac{b_t-c_t}{2}H^2.$$

\noindent When $d^2 +(2t+1)d \geq (t+1)(t+2)n+2$, $b_t < c_t$, and so 
$ a_t <0$. If we let $t = m$, we get the desired result.

\end{proof}
\begin{proposition}\label{h1}
If $E$ is a locally free sheaf on $S$ such that for every irreducible component $C^0$ of a fiber of $\pi$, 
$E|_{C^0}$ is globally generated, then $R^1\pi_*E = 0$.  
\end{proposition}

\begin{proof}
By cohomology and base change \cite[Theorem III.12.11]{hart}, it suffices to prove that for 
every fiber $C$ of $\pi$, $H^1(C, E|_C)=0$. By Lemma \ref{intersection}, we can write 
$C= C_1 +\dots +C_l$ such that every $C_i$ is an irreducible component of $C$ and 
such that $(C_1+\dots +C_i) \cdot C_{i+1} \leq 1$ for every $1 \leq i \leq l-1$. Hence 
$$H^1(C_{i+1}, (E \otimes \mathcal I_{C_1 +\dots +C_i}) |_{C_{i+1}})=0 \,\, {\text{ for every }}0 \leq i \leq l-1.$$ 

On the other hand, for every $0 \leq i \leq l-2$, we have a short exact sequence of $\O_S$-modules 
$$ 
0 \to E \otimes \mathcal I_{C_1+ \dots + C_{i+1}}|_{C_{i+2} +\dots + C_l} \to 
E \otimes \mathcal I_{C_1+ \dots + C_i}|_{C_{i+1}+\dots+C_l} \to  E \otimes \mathcal I_{C_1 + \dots +C_i} |_{C_{i+1}} \to 0.
$$   
\noindent So a decreasing induction on $i$ shows that for every $0 \leq i \leq l-2$, 
$$H^1(S, E \otimes \mathcal I_{C_1+ \dots + C_i}|_{C_{i+1}+\dots+C_l}) = 0.$$
\end{proof}

\begin{lemma}\label{intersection}
Let $C$ be a fiber of $\pi$ with $l$ irreducible components counted with multiplicity. 
Then as a 1-cycle, $C$ can be written as $C_1+ \dots +C_l$ such that 
each $C_i$ is an irreducible 
component of $C$ and for every $1 \leq i \leq l-1$,
$$(C_1+\dots +C_i) \cdot C_{i+1} \leq 1.$$
\end{lemma}

\begin{proof}
The proof is by induction on $l$. If $l=1$, there is nothing to prove. Assume the assertion holds for $k \leq l-1$. There is 
at least one component $C^0$ of $C$ such that $C^0\cdot C^0 = -1$. Let $r$ be the multiplicity of $C^0$ in $C$. 
Blowing down $C^0$, we get a rational surface $S^\prime$ over $\P^1$. Denote by $C'$ 
the blow-down of $C$. Then by the induction hypothesis, we can write $$C^\prime =  C'_1 
+ \dots + C^\prime_{l-r}$$ such that $(C'_1 + \dots + C^\prime_i) \cdot C^\prime_{i+1} \leq 1$ for every $1 \leq i \leq l-r-1$. 
Let $C_i$ be the proper transform of $C'_i$. Then if in the above sum we replace $C^\prime_i$ by  $C_i$ when $C_i$ does not intersect 
$C^0$ and by $C_i + C^0$ when $C_i$ intersects $C^0$, we get the desired result for $C$.
\end{proof}

\begin{proposition}\label{global}
Let $X \subset \P^n$ be a general hypersurface of degree $d$. For any morphism 
$h: \P^1 \to X$, $h^*(T_X  (1))$ is globally generated. 
\end{proposition}

\begin{proof}
The proposition follows from \cite[Proposition 1.1]{voisin}. We give a proof here for the sake of 
completeness. 

Consider the short exact sequence 
$$ 0 \to h^*T_X \to h^*T_{\P^n} \to h^*\O_X(d) \to 0.$$
\noindent Since $X$ is general, the image of the pull-back map 
$H^0(X,\O_X(d)) \to H^0(\P^1, h^*\O_X(d))$ is contained in the image of the map 
$H^0(\P^1, h^*T_{\P^n}) \to H^0(\P^1, h^*\O_X(d)).$ Choose a homogeneous coordinate system for $\P^n$. 
Let $p$ be a point in $\P^1$, and without loss of 
generality assume that $h(p) = (1:0:\dots:0)$. We show that for any $r \in h^*(T_X(1))|_p$,  
there is $\tilde{r} \in H^0(\P^1, h^*(T_X(1)))$ such that $\tilde{r}|_p = r$. 

Consider the exact sequence 
$$
0 \longrightarrow H^0(\P^1, h^* T_X(1)) \longrightarrow H^0(\P^1, h^*T_{\P^n}(1)) 
\stackrel{\phi}{\longrightarrow}  H^0(\P^1, h^*\O_X(d+1)).
$$
Denote by $s$ the image of $r$ in $h^*(T_{\P^n}(1))|_p$. 
There exists $S \in H^0(\P^n, T_{\P^n}(1))$ such that the restriction of 
$\tilde{s}:=h^*(S)$ to $p$ is 
$s$. Denote by $T$ the image of $S$ in $H^0(\P^n, 
\O_{\P^n}(d+1))$, and let $\tilde{t} = h^*(T)$. 
Then $T$ is a form of degree $d+1$ on $\P^n$, and since $\tilde{t}|_p = 0$, 
we can write 
$$ T = x_1 G_1 + \dots + x_n G_n,$$
where the $G_i$ are forms of degree $d$. Our assumption implies that for every $1 \leq i \leq n$, 
there is $\tilde{s}_i \in H^0(\P^1, h^*T_{\P^n})$ such that 
$\phi(\tilde{s}_i) = h^*G_i$. Then
$$\phi(\tilde{s} - h^*(x_1) \tilde{s}_1 -\dots -h^*(x_n)\tilde{s}_n)= 
\tilde{t}-h^*(x_1G_1)-\dots -h^*(x_n G_n)= 0,
$$
 and therefore, 
$\tilde{s} - h^*(x_1) \tilde{s}_1 -\dots -h^*(x_n)\tilde{s}_n$ is the image of some $\tilde{r} 
\in H^0(\P^1, h^*(T_X(1)))$. Since $(\tilde{s} - h^*(x_1) \tilde{s}_1 -\dots -h^*(x_n)\tilde{s}_n)|_p = 
\tilde{s}|_p =s$, we have $\tilde{r}|_p= r$.

\end{proof}

We remark that although for every $e$ and $n$ with $e \geq n+1 \geq 4$, there are smooth non-degenerate rational curves of degree $e$ in $\P^n$ which are not 
$(e-3)$-normal \cite[Theorem 3.1]{glp}, a general smooth rational curve of degree $e$ in a general hypersurface of degree $d$ has possibly a smaller normality. 

For example, let $C$ be a smooth rational curve of degree $e$ in $\P^n$, and write 
$$N_{C/\P^n} =\O_C(a_1) \oplus \dots \oplus \O_C(a_{n-1}).$$
Then it follows from \cite[Proposition 1.2]{glp} that if $e \geq 3$, and if 
$$m = \max{\{a_i+a_j, 1 \leq i < j \leq n-1\}} + 1 - 2e,$$
the curve $C$ is $m$-normal. We expect  that $N_{C/\P^n}$ is as balanced as possible, i.e. 
$|a_i -a_j| \leq 1$ for every $1 \leq i,j \leq n-1$, if $C$ is a general smooth rational curve 
in a general hypersurface of degree $\leq n/2$ and if the degree of $C$ is large enough compared to $n$.

\medskip

\newcommand{\closer}{\vspace{-1.5ex}}

\bigskip
\noindent{Department of Mathematics, Washington University, St. Louis,
MO 63130}\par
\noindent{beheshti@math.wustl.edu}\par


\begin{thebibliography}{99}

\bibitem{roya-jason} R.~Beheshti, J.~Starr, {\em Rational surfaces in index-one Fano hypersurfaces}, 
Journal of Algebraic  Geom., {\bf 17} 2008, no. 2, 225--274.

\bibitem{ij} I.~Coskun and J.~Starr, {\em Rational curves on smooth cubic hypersurfaces}, IMRN 2009. 

\bibitem{ds} A.~J.~de Jong and J.~Starr, {\em Cubic fourfolds and 
spaces of rational curves}, Illinois J. Math., {\bf 48} (2004), 415--450

\bibitem{ds1} A.~J.~de Jong and  J.~Starr, {\em Low degree complete intersections are 
rationally simply connected}, Preprint, 2006.

\bibitem{glp} L.~Gruson, R.~Lazarsfeld, and C.~Peskine, {\em On a theorem of 
Castelnuovo, ad the equations defining space curves}, Invent. Math. {\bf 72} (1983) 491--506.
 
\bibitem{hrs} J.~Harris, M.~Roth and J.~Starr, {\em Rational curves on 
hypersurfaces of 
low degree I}, J. Reine Angew. Math. {\bf 571} (2004), 73--106.

\bibitem{hart}. R.~Hartshorne, {\em Algebraic Geometry}, Springer-Verlag, New York, 1977.

\bibitem{kol}J.~Kollar, {\em Rational Curves on Algebraic Varieties}, 
volume 32 of Ergebnisse der Mathematik 
und ihrer Grenzgebiete, 3. Folge. Springer-Verlag, 1996.

\bibitem{kollar} J.~Kollar, Y.~Miyaoka and S.~Mori, {\em 
Rational connectedness and boundedness of Fano manifolds}, 
J. Diff. Geom. {\bf 36} (1992), 765--779.

\bibitem{reider} I.~Reider, {\em Vector 
bundles of rank $2$ and linear systems on algebraic surfaces}, 
Ann. of Math. (2)  {\bf 127}  (1988),  no. 2, 309--316. 

\bibitem{s4} J.~Starr, {\em Hypersurfaces of low degree are 
rationally $1$-connected}, Preprint, 2004.


\bibitem{kodaira} J.~Starr, {\em The kodaira dimension of spaces of rational  curves on low 
degree hypersurfaces}, arXiv:math/0305432.   

\bibitem{voisin} C.~Voisin, {\em On a conjecture of Clemens on rational curves on hypersurfaces}, 
J. Differential Geometry {\bf 44} (1996) 200--14.
\end{thebibliography}
\end{document}